\documentclass[a4paper, 10pt, notitlepage]{article}

\usepackage{amsthm, amsmath, amssymb, latexsym}
\usepackage{mathrsfs}
\usepackage[multiple]{footmisc}

\theoremstyle{plain}
\newtheorem{theorem}{Theorem}

\newtheorem{corollary}{Corollary}
\newtheorem{proposition}{Proposition}

\theoremstyle{definition}

\theoremstyle{remark}
\newtheorem{remark}{Remark}

\author{Dolgopolik M.V.\footnote{Saint Petersburg State University, Saint Petersburg, Russia}
\footnote{Institute of Problems of Mechanical Engineering, Saint Petersburg, Russia}}
\title{New Direct Numerical Methods for Some Multidimensional Problems of the Calculus of Variations}

\begin{document}

\maketitle

\begin{abstract}
In this paper we develop a new approach to the design of direct numerical methods for multidimensional problems of 
the calculus of variations. The approach is based on a transformation of the problem with the use of a new class of
Sobolev-like spaces that is studied in the article. This transformation allows one to analytically compute 
the direction of steepest descent of the main functional of the calculus of variations with respect to a certain inner
product, and, in turn, to construct new direct numerical methods for multidimensional problems of the calculus of
variations. In the end of the paper we point out how the approach developed in the article can be extended to the case
of problems with more general boundary conditions, problems for functionals depending on higher order derivatives, and
problems with isoperimetric and/or pointwise constraints.
\end{abstract}

\section{Introduction}

The main problem of the calculus of variations has the form
\begin{equation} \label{Intro_MainProb}
  \min \: \mathcal{I}(u) = \int_{\Omega} f(x, u(x), \nabla u(x)) \, dx \quad \text{subject to}
  \quad u|_{\partial \Omega} = \psi,
\end{equation}
where $\Omega \subset \mathbb{R}^n$ is an open set, $\partial \Omega$ is the boundary of $\Omega$, and
$f \colon \Omega \times \mathbb{R} \times \mathbb{R}^n \to \mathbb{R}$ and 
$\psi \colon \partial \Omega \to \mathbb{R}$ are given functions. Various aspects of this problem, such as the existence
and regularity of solutions \cite{Morrey,Giaquinta,Giusti,DarkSide,Dacorogna}, qualitative properties of critical
points (so-called ``the calculus of variations in the large'') \cite{Morse,Ekeland}, and necessary and sufficient
conditions for a local minimum \cite{GiaqHild}, have been extensively studied by many researches. However, 
relatively little attention has been paid to the development of direct numerical methods (especially in the
multidimensional case) for problems of the calculus of variations. 

By direct numerical methods we mean methods that are not based on direct numerical solution of the Euler--Lagrange
equation, and, instead, rely on the variational formulation of the problem. Besides being one of the approaches to 
numerical solution of some partial differential equations, direct numerical methods are especially useful and important
for those problems of the calculus of variations that arise directly as optimization problems. This type of problems
naturally appear, in particular, in image processing \cite{ImSegm,ImageProc}.

The vast majority of direct numerical methods of the calculus of variations is based on an approximate reduction of 
problem \eqref{Intro_MainProb} to a finite dimensional optimization problem. Various types of reduction techniques
and corresponding numerical methods in the one dimensional case (i.e. in the case when $\Omega = (a, b) \subset
\mathbb{R}$) were proposed in
\cite{Chernousko,Gregory,Gregory1,Gregory2,MaWalk,ZhouChalabi1,ZhouChalabi2,Elnagar,Levin,Agrawal,NeuralNet,Amini}. In
the multidimensional case, the range of choice of direct methods is much more narrow, and it
includes (although is not exhausted by) the finite elements methods, the Galerkin method and the Ritz method
\cite{Mikhlin,Reddy,Komzsik,Cassel,Carstensen,Ortner,NgSancho,Banichuk}. However, there exist some direct numerical
methods of the calculus of variations in the one dimensional case that do not consist of an approximate reduction to a
finite dimensional problem. Among them are the first- and second-variation methods \cite{Miele}, that are based on
straightforward usage of the necessary optimality conditions for problem \eqref{Intro_MainProb}, He's variational
iteration method \cite{Yousefi}, the continuous method of steepest descent \cite{Rosenbloom}, the discrete steepest
descent method \cite{Stein,Polyak}, Newton's method \cite{NewtonMethod}, and the method of hypodifferential descent
based on the use of exact penalty functions \cite{DemTam,Tamsyan1,DemTam2}. Let us also mention that some
multidimensional problems of the calculus of variations can be solved by standard gradient based methods for
functionals defined on Hilbert or Banach spaces \cite{Rosenbloom,Polyak,Daniel,Vainberg1,DemRub70,
Vainberg2,KantAkilov,Penot} with the use of the so-called Sobolev gradients \cite{Neuberger}.

The main goal of this paper is to develop a new approach to the design of direct numerical methods for
multidimensional problems of the calculus of variations. This approach is based on a transformation of problem
\eqref{Intro_MainProb} that allows one to analytically compute the direction of steepest descent of the functional
$\mathcal{I}$ with respect to a certain norm. Utilizing this direction of steepest descent one can apply an
obvious modification of almost any first order method of finite dimensional optimization to problem
\eqref{Intro_MainProb}. The basic ideas (in a very crude form) that lead to the the development of the approach studied
in this article were presented in the two dimensional case in \cite{DolgTam}. It should be noted that the main results
of the paper \cite{DolgTam} were inspired by the ideas of the late prof. V.F.~Demyanov \cite{VFD94,Demyanov}, as well as
the method of hypodifferential descent and the method of steepest descent mentioned above.

The paper is organised as follows. In Section~\ref{Section_InformalArg} we informally discuss the underlying ideas of
the approach developed in this article. In Section~\ref{Section_FuncSpaces} we introduce and study a new class of
Sobolev-like spaces that plays a central role in the formalization of the new direct numerical methods. In particular,
we obtain a convenient characterization of a certain function space from this class that is very important for 
the transformation of problem \eqref{Intro_MainProb}. The direction of steepest descent of the functional
$\mathcal{I}$ with respect to a certain norm is derived in Section~\ref{Section_NewMinMethods}. In the Conclusion we
discuss possible generalizations of the ideas developed in this paper, and briefly outline some directions of future
research.

\section{How to Compute the Direction of Steepest Descent?}
\label{Section_InformalArg}

In this section, we informally discuss a general technique for constructing new minimization methods for the main
problem of the calculus of variations. A possible formalization of this technique is presented in the subsequent
sections. We suppose that all functions that appear in this section are sufficiently smooth.

Consider the main problem of the calculus of variations
\begin{equation} \label{EqMainProbTDim}
  \min \: \mathcal{I}(u) = \int_{\Omega} f(x, u(x), \nabla u(x)) \, dx \quad \text{subject to}
  \quad u|_{\partial \Omega} = \psi,
\end{equation}
where $\Omega \subset \mathbb{R}^n$ is an open set and $f = f(x, u, z)$. We want to apply infinite dimensional
analogues of standard gradient-based methods of finite dimensional optimization to the problem above. In order to do
that we need to compute the gradient or, more generally, the direction of steepest descent of the functional
$\mathcal{I}$.

It is well-known and easy to check that the functional $\mathcal{I}$ is G\^ateaux differentiable, and its G\^ateaux
derivative has the form
\begin{equation} \label{GatDerIP}
  \mathcal{I}'[u](h) = \int_{\Omega} \left( \frac{\partial f}{\partial u}(x, u(x), \nabla u(x)) h(x) +
  \sum_{i = 1}^n \frac{\partial f}{\partial z_i}(x, u(x), \nabla u(x)) \frac{\partial h}{\partial x_i}(x) 
  \right) \, dx.
\end{equation}
Therefore in order to compute the direction of steepest descent of the functional $\mathcal{I}$ we need to solve
the following problem of the calculus of variations
$$
  \min \: \mathcal{I}'[u](h) \quad \text{subject to} \quad 
  h|_{\partial \Omega} = 0, \quad \| h \| \le 1,
$$
where $\| \cdot \|$ is some norm. However, this problem is, usually, too complicated to be solved analytically, and
even in simple cases it is equivalent to the problem of solving a linear partial differential equation (see, e.g.,
\cite{Neuberger}, Chapter~9).

In order to overcome this difficulty, let us transform problem \eqref{EqMainProbTDim}. For the sake of simplicity,
suppose that $n = 2$, $\psi \equiv 0$, and let $\Omega$ be an open box, i.e. $\Omega = (a_1, b_1) \times (a_2, b_2)$. 
Let also $u$ be a function such that $u|_{\partial \Omega} = 0$. Observe that
\begin{equation} \label{NewtLebForm2dim}
  u(x_1, x_2) = \int_{a_1}^{x_1} \frac{\partial u}{\partial x_1}(\xi_1, x_2) \, d \xi_1 \quad \forall x \in \Omega,
\end{equation}
and $\partial u / \partial x_1(x_1, a_2) \equiv 0$ due to the fact that $u(x_1, a_2) \equiv 0$. Therefore
$$
  \frac{\partial u}{\partial x_1}(x_1, x_2) = 
  \int_{a_2}^{x_2} \frac{\partial^2 u}{\partial x_2 \partial x_1}(x_1, \xi_2) \, d \xi_2 \quad \forall x \in \Omega.
$$
Hence with the use of \eqref{NewtLebForm2dim} one gets that
$$
  u(x_1, x_2) = \int_{a_1}^{x_1} \int_{a_2}^{x_2} 
  \frac{\partial^2 u}{\partial x_2 \partial x_1}(\xi_1, \xi_2) \, d \xi_2 d \xi_1 \quad \forall x \in \Omega.
$$
Furthermore, from the latter equality it follows that
$$
  \int_{a_1}^{b_1} \frac{\partial^2 u}{\partial x_2 \partial x_1}(\xi_1, \cdot) \, d \xi_1 \equiv 0, \quad
  \int_{a_2}^{b_2} \frac{\partial^2 u}{\partial x_2 \partial x_1}(\cdot, \xi_2) \, d \xi_2 \equiv 0
$$
due to the fact that $u(x_1, b_2) \equiv 0$ and $u(b_1, x_2) \equiv 0$.

Let now $v$ be a function such that
$$
  \int_{a_1}^{b_1} v(\xi_1, \cdot) \, d \xi_1 \equiv 0, \quad
  \int_{a_2}^{b_2} v(\cdot, \xi_2) \, d \xi_2 \equiv 0.
$$
Then it is easy to see that for the function $u = T v$, where
$$
  (T v)(x_1, x_2) = 
  \int_{a_1}^{x_1} \int_{a_2}^{x_2} v(\xi_1, \xi_2) \, d \xi_2 d \xi_1 \quad \forall x \in \Omega,
$$
one has $u|_{\partial \Omega} = 0$. Thus, we have that the following result holds true.

\begin{proposition} \label{PrpChrctFunc}
Let $u \colon [a_1, b_1] \times [a_2, b_2] \to \mathbb{R}$ be a sufficiently smooth function. Then 
$u|_{\partial \Omega} = 0$ if and only if there exists a sufficiently smooth function $v$ such that
\begin{enumerate}
\item{$u(x_1, x_2) = (T v)(x_1, x_2) = \int_{a_1}^{x_1} \int_{a_2}^{x_2} v(\xi_1, \xi_2) \, d \xi_2 d \xi_1$ for all 
$x \in \Omega$;}

\item{$\int_{a_1}^{b_1} v(\xi_1, \cdot) \, d \xi_1 \equiv 0$ and 
$\int_{a_2}^{b_2} v(\cdot, \xi_2) \, d \xi_2 \equiv 0$.
}
\end{enumerate}
Moreover, $v = \partial^2 u / \partial x_1 \partial x_2$.
\end{proposition}

From the proposition above it follows that problem \eqref{EqMainProbTDim} with $\psi \equiv 0$ is equivalent to the
following optimization problem: minimize
\begin{equation} \label{ProblemRefom}
  F(v) = \int_{\Omega} f\Big(x, (T v)(x), 
  \int_{a_2}^{x_2} v(x_1, \xi_2) \, d \xi_2,  \int_{a_1}^{x_1} v(\xi_1, x_2) \, d \xi_1 \Big) \, d x_1 d x_2
\end{equation}
subject to the constraints
\begin{equation} \label{Constraints}
  \int_{a_1}^{b_1} v(\xi_1, \cdot) \, d \xi_1 \equiv 0, \quad
  \int_{a_2}^{b_2} v(\cdot, \xi_2) \, d \xi_2 \equiv 0.
\end{equation}
As we shall see, one can easily compute the direction of steepest descent for this problem.

Indeed, denote by $L_0$ the linear space consisting of all functions $v$ satisfying \eqref{Constraints}. Clearly, the
functional $F$ is G\^ateaux differentiable. Integrating by parts in \eqref{GatDerIP} one gets that the G\^ateaux
derivative of $F$ has the form
$$
  F'[v](h) = \int_{\Omega} Q(v)(x) h(x) \, dx,
$$
where
\begin{multline} \label{QofV}
  Q(v)(x) = 
  \int_{x_1}^{b_1} \int_{x_2}^{b_2} \frac{\partial f}{\partial u}(\xi, u(\xi), \nabla u(\xi)) \, d \xi_2 d \xi_1 - \\
  - \int_{x_2}^{b_2} \frac{\partial f}{\partial z_1}(x_1, \xi_2, u(x_1, \xi_2), \nabla u(x_1, \xi_2)) \, d \xi_2 - \\
  - \int_{x_1}^{b_1} \frac{\partial f}{\partial z_2}(\xi_1, x_2, u(\xi_1, x_2), \nabla u(\xi_1, x_2)) \, d \xi_1,
\end{multline}
and $u = T v$. Hence the direction of steepest descent for problem \eqref{ProblemRefom}, \eqref{Constraints} is a
solution of the following optimization problem:
\begin{equation} \label{ProbDSD_Mod}
  \min \: \int_{\Omega} Q(v)(x) h(x) \, dx \quad \text{subject to} 
  \quad h \in L_0, \quad \| h \| \le 1,
\end{equation}
where $\| \cdot \|$ is some norm. We choose the $L_2$-norm, i.e. 
$\| h \| = \Big( \int_{\Omega} h^2(x) \, dx \Big)^{\frac12}$.

Let us solve problem \eqref{ProbDSD_Mod}. 

\begin{proposition} \label{Prp_SDD_Informal}
Suppose that the function $u = T v$ does not satisfy the Euler--Lagrange equation for the functional $\mathcal{I}$.
Then the direction of steepest descent $h^*$ for problem \eqref{ProblemRefom}, \eqref{Constraints} has the form 
$h^*(x) = G(v)(x) / \| G(v) \|_2$, where
\begin{multline} \label{DSD_informal}
  G(v)(x) = - Q(v)(x) + \frac{1}{b_1 - a_1} \int_{a_1}^{b_1} Q(v)(\xi_1, x_2) d \xi_1 +  \\
  + \frac{1}{b_2 - a_2} \int_{a_2}^{b_2} Q(v)(x_1, \xi_2) d \xi_2 -
  \frac{1}{(b_1 - a_1)(b_2 - a_2)} \int_{a_1}^{b_1} \int_{a_2}^{b_2} Q(v)(\xi_1, \xi_2) d \xi_2 d \xi_1.
\end{multline}
\end{proposition}

\begin{proof}
Applying the Lagrange multipliers rule to problem \eqref{ProbDSD_Mod} one gets that there exists 
$\lambda \in \mathbb{R}$ such that
$$
  \int_{\Omega} (Q(v)(x) + \lambda h^*(x)) h(x) \, dx = 0 \quad \forall h \in L_0.
$$
Note that for any infinitely differentiable function $\varphi$ with compact support one has 
$\partial^2 \varphi / \partial x_1 \partial x_2 \in L_0$ (see Proposition~\ref{PrpChrctFunc}). Therefore
$$
  \int_{\Omega} (Q(v)(x) + \lambda h^*(x)) \frac{\partial^2 \varphi}{\partial x_1 \partial x_2}(x) \, dx =
  \int_{\Omega} \frac{\partial^2 }{\partial x_1 \partial x_2} (Q(v)(x) + \lambda h^*(x)) \varphi(x) \, dx = 0
$$
for any infinitely differentiable function $\varphi$ with compact support. Hence applying the fundamental lemma of
the calculus of variations one gets that
$$
  \frac{\partial^2 }{\partial x_1 \partial x_2} (Q(v)(x) + \lambda h^*(x)) = 0 \quad \forall x \in \Omega.
$$
It is easy to verify that $\lambda = 0$ if and only if the function $u = T v$ satisfies the Euler--Lagrange equation
for the functional $\mathcal{I}$ (see~\eqref{QofV}). Thus, we can suppose that $\lambda \ne 0$. Hence
$$
  h^*(x) = - \frac{1}{\lambda} Q(v)(x) + r_1(x_1) + r_2(x_2),
$$
where $r_1$ and $r_2$ are some functions. Taking into account the fact that $h^* \in L_0$ one obtains that
$$
  \begin{cases}
  - \frac{1}{\lambda} \int_{a_1}^{b_1} Q(v)(\xi_1, \cdot) d \xi_1 + \int_{a_1}^{b_1} r_1(\xi_1) d \xi_1 +
  (b_1 - a_1) r_2(\cdot) = 0, \\
  - \frac{1}{\lambda} \int_{a_2}^{b_2} Q(v)(\cdot, \xi_2) d \xi_2 + (b_2 - a_2) r_1(x_1) +
  \int_{a_2}^{b_2} r_2(\xi_2) d \xi_2 = 0
  \end{cases}
$$
Solving this system with respect to $r_1$ and $r_2$ one obtains that
\begin{gather*}
  r_1(x_1) = \frac{1}{\lambda(b_2 - a_2)} \int_{a_2}^{b_2} Q(v)(x_1, \xi_2) d \xi_2, \\
  r_2(x_2) = \frac{1}{\lambda(b_1 - a_1)} \int_{a_1}^{b_1} Q(v)(\xi_1, x_2) d \xi_1 -
  \frac{1}{\lambda (b_1 - a_1)(b_2 - a_2)} \int_{\Omega} Q(v)(\xi) d \xi.
\end{gather*}
Hence \eqref{DSD_informal} holds true.
\end{proof}

Since we know the direction of steepest descent for problem \eqref{ProblemRefom}, \eqref{Constraints}, we can apply
an obvious modification of almost any gradient-based algorithm of finite dimensional optimization to this problem and,
in turn, to the initial problem \eqref{EqMainProbTDim}. 

If we look at the way the direction of steepest descent was derived, we can easily see that this direction is the
direction of the steepest descent of the functional $\mathcal{I}$ with respect to the norm
$$
  \| u \| =  
  \left( \int_{\Omega} \left( \frac{\partial^2 u}{\partial x_1 \partial x_2}(x) \right)^2 dx \right)^{\frac12}
$$
(the fact that this seminorm is, indeed, a norm follows from Proposition~\ref{PrpChrctFunc}). However, any space of
smooth functions equipped with this norm is incomplete. Therefore it is natural to consider the original problem in the
setting of Sobolev-like spaces, and to transfer the main ideas discussed above to this more general setting.

\begin{remark}
As it is well-known, the direction of steepest descent as well as the performance of the method of steepest descent
depend on the choice of an underlying Hilbert (Banach) space and an inner product (norm) in this space 
(see, e.g., \cite{Nashed65,Nashed71,Neuberger}). From this point of view, the main goal of this article is to introduce
a Hilbert space such that the direction of the steepest descent of the functional $\mathcal{I}(u)$ in this space can be
easily computed analytically.
\end{remark}

\section{Special Function Spaces}
\label{Section_FuncSpaces}

In this section, we introduce a class of function spaces that plays a central role in the formalization of 
the minimization methods for multidimensional problems of the calculus of variations discussed above. This class of
functions is closely related to the Sobolev spaces and possesses many properties of these spaces. We suppose that 
the reader is familiar with basic results on the Sobolev spaces that can be found in \cite{Adams, Evans, Leoni}.

\subsection{Main Definitions and Basic Properties}

Introduce the notation first. A point in $\mathbb{R}^n$ is denoted by $x = (x_1, \ldots, x_n) \in \mathbb{R}^n$, and its
norm is denoted by $|x| = (\sum_{i = 1}^n x_i^2)^{\frac{1}{2}}$. As usual, any $n$-tuple  $\alpha = (\alpha_1, \ldots,
\alpha_n) \in Z^n_+$ of nonnegative integers $\alpha_i$ is called a multi-index; its absolute value 
$|\alpha| = \alpha_1 + \ldots + \alpha_n$. For any multi-index $\alpha$ denote by 
$D^{\alpha} = D_1^{\alpha_1} \dots D_n^{\alpha_n}$ a differential operator of order $|\alpha|$, where 
$D_i = \partial / \partial x_i$  for $i \in \{1, \ldots, n\}$. If $\alpha = (0, \ldots, 0)$, then $D^{\alpha} u = u$ for
any function $u$. Define
$$
  I_k = \big\{ \alpha \in Z^n_+ \mid 
  |\alpha| = k, \alpha_i = 0 \text{ or } \alpha_i = 1 \; \forall i \in \{ 1, \ldots, n \} \big\}.
$$
for any $k \in \{ 0, \ldots, n \}$. It is clear that $I_0 = \{ (0, \ldots, 0) \}$ and $I_n = \{ (1, \ldots, 1) \}$. If
$\alpha \in I_k$ for some $0 \le k \le n$, then a unique multi-index $\beta \in I_{n - k}$ such that 
$\alpha + \beta \in I_n$ is denoted by $\overline{\alpha}$. Here the sum of multi-indices is component-wise. The kernel
of a linear operator $L \colon X \to Y$ is referred to as $\ker L$ (here $X$, $Y$ are linear spaces).

\begin{remark}
We consider only real valued functions and normed spaces over the field of real numbers. If $f$ is a bounded linear
functional defined on a normed space $X$, then we denote its norm by $\| f \|$ or by $\| f \|_X$ when we want to
specify the domain of $f$.
\end{remark}

Hereafter, let $\Omega \subset \mathbb{R}^n$ be a bounded open box, i.e. $\Omega = \prod_{k = 1}^n (a_i, b_i)$. Denote
by $C^k(\overline{\Omega})$ the set of all those $u \in C^k(\Omega)$ for which all functions $D^{\alpha} u$ with 
$0 \le |\alpha| \le k$ are bounded and uniformly continuous on $\Omega$ (then there exist unique continuous extensions
of functions $D^{\alpha} u$ with $0 \le |\alpha| \le k$ to the closure $\overline{\Omega}$ of the set $\Omega$).  The
set of all infinitely continuously differentiable functions $u \colon \Omega \to \mathbb{R}$ with compact support is
denoted by $C_0^{\infty}(\Omega)$.

Let us introduce a new function space. For any $m \in \{1, \ldots, n\}$ and  $1 \le p \le \infty$ denote by 
$M^{m, p}(\Omega)$ the set of all $u \in L_p(\Omega)$ such that for any  $k \in \{1, \ldots, m\}$ and $\alpha \in I_k$
there exists the weak derivative $D^{\alpha} u$ belonging to $L_p(\Omega)$. Thus, $M^{m, p}(\Omega)$ consists of all
functions $u \in L_p(\Omega)$ for which there exist all weak mixed derivatives of the order $k = 1 \colon m$ that belong
to $L_p(\Omega)$.  The set $M^{m, p}(\Omega)$ is a linear space that can be equipped with the norm
$$
  \| u ; M^{m, p} \| = 
  \begin{cases}
    \left( \sum_{k = 0}^m \sum_{\alpha \in I_k} (\| D^{\alpha} u \|_p)^p \right)^{\frac 1p} 
    \text{ for } 1 \le p < \infty, \\ 
    \max\{\| D^{\alpha} u \|_{\infty} \mid \alpha \in I_k, 1 \le k \le m \} \text{ for } p = \infty,
  \end{cases}
$$
where $\| \cdot \|_p$ is the standard norm on $L_p(\Omega)$. The closure of $C^{\infty}_0(\Omega)$ in the normed space
$M^{m, p}(\Omega)$ is denoted by $M^{m, p}_0(\Omega)$.

Let, as usual, $W^{m, p}(\Omega)$ with $m \in \mathbb{N}$ and $1 \le p \le \infty$ be the Sobolev space, and 
$W_0^{m, p}(\Omega)$ be the closure of $C_0^{\infty}(\Omega)$ in $W^{m, p}(\Omega)$. It is clear that 
$W^{m, p}(\Omega) \subset M^{m, p}(\Omega)$ and, analogously, $W^{m, p}_0(\Omega) \subset M^{m, p}_0(\Omega)$. Moreover,
these embeddings are continuous. On the other hand, $M^{m, p}(\Omega)$ is dense in $W^{1, p}(\Omega)$ for any 
$m \in \{ 1, \ldots, n \}$, and $M^{1, p}(\Omega) = W^{1,p}(\Omega)$.

Let us describe some properties of the spaces $M^{m, p}(\Omega)$. Arguing in a similar way to the case of the Sobolev
spaces (see, e.g., theorems 3.2 and 3.5 in \cite{Adams}) one can easily derive the following results.

\begin{theorem}
For any $1 \le m \le n$ the space $M^{m, p}(\Omega)$ is complete in the case $1 \le p \le \infty$, is separable in the
case $1 \le p < \infty$, and is reflexive and uniformly convex in the case $1 < p < \infty$. 
Moreover, $M^{m, 2}(\Omega)$ is a separable Hilbert space with the inner product
$$
  \langle u, v \rangle_m = \sum_{k = 0}^m \sum_{\alpha \in I_k} \langle D^{\alpha} u, D^{\alpha} v \rangle,
$$
where $\langle \varphi, \psi \rangle = \int_{\Omega} \varphi(x) \psi(x) \, dx$ is the inner product in $L_2(\Omega)$.
\end{theorem}

\begin{remark} \label{Rmrk_Multip}
{(i) An analogous result holds true for $M^{m, p}_0(\Omega)$.}

\noindent{(ii) Note that for any $\varphi \in C^{\infty}(\Omega) \cap M^{m, \infty}(\Omega)$ and 
$u \in M^{m, p}(\Omega)$ one has $\varphi u \in M^{m, p}(\Omega)$. Furthermore, for any $\alpha \in I_k$, 
$1 \le k \le m$ one has $$
  D^{\alpha} (\varphi u) = \sum_{\beta + \gamma = \alpha} D^{\beta} \varphi D^{\gamma} u.
$$
}
\end{remark}

It is well-known that the space $W^{m, p}_0(\Omega)$ can be equipped with the norm
\begin{equation} \label{EquivNormOnW12_0}
  \| u \|_{0, m, p} = \Big( \sum_{|\alpha| = m} (\| D^{\alpha} u \|_p)^p \Big)^{\frac 1p}, 
  \quad u \in W^{m, p}_0(\Omega)
\end{equation}
which is equivalent to the standard norm $\| \cdot \|_{m, p}$ (see, e.g., \cite{Adams}, sections 6.25--6.26).
Analogously, the space $M^{m, p}_0 (\Omega)$ can be equipped with a different norm, which is equivalent to the norm 
$\| \cdot; M^{m, p} \|$, and is more suitable for our purposes. Set 
$$
  \| u; M^{m, p}_0 \| = \Big( \sum_{\alpha \in I_m} (\| D^{\alpha} u \|_p)^p \Big)^{\frac 1p},
  \quad u \in M^{m, p}_0(\Omega).
$$
It is clear that $\| \cdot; M^{m, p}_0 \|$ is a seminorm on $M^{m, p}_0(\Omega)$. Arguing in a similar way to the case
of $W^{m, p}_0(\Omega)$ (cf.~\cite{Adams}, section 6.26) one can easily verify that the following theorem holds true.

\begin{theorem} \label{ThEquivNorm}
The seminorm $\| \cdot; M^{m, p}_0 \|$ is a norm on $M^{m, p}_0(\Omega)$ which is equivalent to the norm 
$\| \cdot; M^{m, p} \|$.
\end{theorem}

\subsection{A Characterization of $M^{n, 2}_0 (\Omega)$}

Our aim now is to give a simple characterization of the space $M^{n, 2}_0(\Omega)$ that is crucial for the computation
of the direction of steepest descent. In order to do that we need to introduce several integral operators that will be
useful in the sequel. 

Let $v \in L_2(\Omega)$. For $1 \le i \le n$ define the operator
$$
  (T_i v)(x) = \int_{a_i}^{x_i} v(x_1, \ldots x_{i-1}, \xi_i, x_{i+1}, \ldots, x_n) \, d \xi_i 
  \quad \text{for a.e.~} x \in \Omega.
$$
By virtue of the Fubini theorem one gets that $T_i$ is correctly defined and is a continuous linear operator mapping
$L_2(\Omega)$ to $L_2(\Omega)$. Let $\alpha \in I_k$, $1 \le k \le n$, and suppose that $\alpha_{i_j} = 1$, 
$1 \le j \le k$, where $1 \le i_1 < \ldots < i_k \le n$, and $\alpha_l = 0$ iff $l \ne i_j$ for any $1 \le j \le k$,
i.~e. $\alpha_{i_j}$ are the only nonzero components of the multi-index $\alpha$. Then define the operator
$T_{\alpha} = T_{\alpha_{i_1}} \circ \ldots \circ T_{\alpha_{i_k}}$ mapping continuously $L_2(\Omega)$ to $L_2(\Omega)$.
With the use of the Fubini theorem one gets that for any permutation $\ell$ of the set $\{1, \ldots, k \}$ the following
holds
$$
T_{\alpha} = T_{\alpha_{ i_{\ell(1)} }} \circ \ldots \circ T_{\alpha_{ i_{\ell(k)} }}.
$$
For the sake of convenience, denote $T = T_{(1, \ldots, 1)}$. Applying the Lebesgue differentiation theorem (cf., for
instance, \cite{Federer}, corollary 2.9.9) and integrating by parts one can verify that for any $v \in L_2(\Omega)$, 
$1 \le k \le n$ and $\alpha \in I_k$ there exists the weak derivative $D^{\alpha} (T v) \in L_2(\Omega)$ and
\begin{equation} \label{DerOfIntegOper}
  D^{\alpha} (T v) = T_{\overline{\alpha}} (v).
\end{equation}
Recall that $\overline{\alpha} \in I_{n - k}$ is a unique multi-index such that $\alpha + \overline{\alpha} \in I_n$.
Thus, as it easy to see, $T$ is a continuous linear operator mapping $L_2(\Omega)$ to $M^{n, 2}(\Omega)$.

For any $\alpha \in I_k$, $1 \le k \le n$ and for all $v \in L_2(\Omega)$ define
$$
  S_{\alpha} v = \int_{ a_{\alpha_{i_1}} }^{ b_{\alpha_{i_1}} } \ldots \int_{ a_{\alpha_{i_k}} }^{ b_{\alpha_{i_k}} }
  v \, d \xi_{ \alpha_{i_k} } \ldots d \xi_{ \alpha_{i_1} },
$$
where $1 \le i_1 < \ldots < i_k \le n$ and $\alpha_{r} = 0$ iff $r \ne i_j$ for all $1 \le j \le k$. For the sake of
convenience denote $S_i = S_{\alpha}$ for any $\alpha \in I_1$, where $i$ is the only non-zero component of $\alpha$.

It is easy to verify that for any $\alpha \in I_k$, $1 \le k \le n$ the linear operator $S_{\alpha}$ continuously maps
$L_2(\Omega)$ to $L_2(\Omega)$. Therefore, in particular, the linear subspace 
$$
  L_0 = \bigcap_{1 \le i \le n} \ker S_i = 
  \bigg\{ v \in L_2(\Omega) \biggm| \int_{a_k}^{b_k} v \, d \xi_k = 0 \: k \in \{ 1, \ldots, n \} \bigg\}.
$$
of the Hilbert space $L_2(\Omega)$ is closed. Consequently, there exists the orthogonal projector $Pr_{L_0}$ of
$L_2(\Omega)$ onto $L_0$. We will need an explicit formula for the projector $Pr_{L_0}$
(cf.~Proposition~\ref{Prp_SDD_Informal}).

\begin{proposition} \label{Prp_AnalyticalProjection}
For any $v \in L_2(\Omega)$ one has
\begin{equation} \label{ProjExplForm}
  Pr_{L_0} v = v + \sum_{k = 1}^n \sum_{\alpha \in I_k} (-1)^{|\alpha|} c_{\alpha} S_{\alpha} v,
\end{equation}
where $c_{\alpha} = \prod_{i = 1}^n (b_i - a_i)^{- \alpha_i}$.
\end{proposition}

\begin{proof}
Fix an arbitrary $v \in L_2(\Omega)$, and denote the function on the right-hand side of \eqref{ProjExplForm} by $w$. 
A direct computation shows that $\int_{a_k}^{b_k} w \, d \xi_k = 0$ for all $k \in \{1, \ldots, n\}$.
Consequently, $w \in L_0$. Let us show that $v - w \in L_0^{\perp}$, where $L_0^{\perp}$ is the orthogonal
complement of $L_0$, then $Pr_{L_0} v = w$, since the decomposition $v = v_1 + v_2$ for $v_1 \in L_0$ and 
$v_2 \in L_0^{\perp}$ is unique. For an arbitrary $h \in L_0$ one has
$$
  \langle v - w, h \rangle = \sum_{k = 1}^n \sum_{\alpha \in I_k} (-1)^{|\alpha|} c_{\alpha} 
  \int_{\Omega} (S_{\alpha} v)(x)  h(x) \, dx.
$$
For any $\alpha \in I_k$, $1 \le k \le n$ there exists $1 \le i \le n$ such that $\alpha_i = 1$. 
Hence $S_{\alpha} v$ does not depend on $x_i$. Denote $\Omega_i = \prod_{k \ne i} (a_i, b_i)$. One has
$$
  \int_{\Omega} (S_{\alpha} v)(x)  h(x) \, dx = \int_{\Omega_i} (S_{\alpha} v)(x) (S_i h)(x) \, 
  d x_1 \ldots d x_{i-1} d x_{i+1} \ldots d x_n   = 0,
$$
by the Fubini theorem and the fact that $S_i h = 0$, since $h \in L_0$. Therefore
$\langle v - w, h \rangle = 0$ for any $h \in L_0$, which means that $v - w \in L_0^{\perp}$. 
\end{proof}

\begin{remark}
{(i) Note that for any function $\varphi \in C^n(\Omega)$ with compact support one has 
$D^{(1, \ldots, 1)} \varphi \in L_0$.
}

\noindent{(ii) It is obvious that if $v \in C^k(\Omega) \cap L_2(\Omega)$ then $Pr_{L_0} v \in C^k (\Omega)$, i.~e. the
projection operator $Pr_{L_0}$ preserves smoothness.
}
\end{remark}

The following theorem gives a convenient characterization of $M^{n, 2}_0(\Omega)$ (cf.~Proposition~\ref{PrpChrctFunc}).

\begin{theorem} \label{ThCharctOfMd2}
A function $u \colon \Omega \to \mathbb{R}$ belongs to $M^{n, 2}_0(\Omega)$ if and only if there exists a function 
$v \in L_2(\Omega)$ such that
\begin{enumerate}
\item{$v \in L_0$, i.~e. for any $1 \le i \le n$ one has $S_i v = \int_{a_i}^{b_i} v \, d \xi_i = 0$,
}

\item{$u(x) = (T v)(x) = \int_{a_n}^{x_n} \ldots \int_{a_1}^{x_1} v(\xi) d \xi_1 \ldots d \xi_n $ for a.e. 
$x \in \Omega$.}
\end{enumerate}
\end{theorem}

\begin{proof}
Let us show that for any $u \in M^{n, 2}_0(\Omega)$ one has that $v = D^{(1, \ldots, 1)} u \in L_0$ and $u = T v$.
Indeed, it is clear that for any $\varphi \in C^{\infty}_0(\Omega)$ one has that $D^{(1, \ldots, 1)} \varphi \in L_0$
and $\varphi = T D^{(1, \ldots, 1)} \varphi$. Let $u \in M^{n, 2}_0(\Omega)$ be
arbitrary, and $\{ \varphi_k \} \subset C^{\infty}_0(\Omega)$ be a sequence such that $\varphi_k \to u$ in 
$M^{n, 2}(\Omega)$. Then, denoting $v = D^{(1, \ldots, 1)} u$ and $v_k = D^{(1, \ldots, 1)} \varphi_k$, one gets that
for some $C \ge 0$, depending only on $n$ and $\Omega$, the following inequalities hold true
\begin{multline*}
  \| u - T v; M^{n, 2}(\Omega) \| \le \| u - \varphi_k; M^{n, 2}(\Omega) \| +  
  \| T v - T v_k; M^{n, 2}(\Omega) \| \le \\
  \le \| u - \varphi_k; M^{n, 2}(\Omega) \| + C \| v - v_k \|_2 
  \le (C + 1) \| u - \varphi_k; M^{n, 2}(\Omega) \| \to 0
\end{multline*}
as $k \to \infty$. Consequently, $u = T (D^{(1, \ldots, 1)} u)$. Moreover, as it easy to verify, $S_i v = 0$ for all 
$1 \le i \le n$, since $S_i v_k = 0$ for any $k \in \mathbb{N}$. Consequently, $v = D^{(1, \ldots, 1)} u \in L_0$.

Let us prove the converse statement. Fix an arbitrary $v \in L_0$, and set $u = T v$. We need to prove that there
exists a sequence $\{ \varphi_m \} \subset C^{\infty}_0(\Omega)$ such that $\varphi_m \to u$ in $M^{n, 2}(\Omega)$. We
will prove that there exist functions $u_m \in M^{n, 2}(\Omega)$ such that $u_m \to u$ in $M^{n, 2}(\Omega)$ and 
$u_m = 0$ outside some compact set $K_m \subset \Omega$. Then one can mollify $u_m$ to generate a sequence 
$\{ \varphi_m \} \subset C^{\infty}_0 (\Omega)$ such that $\varphi_m \to u$ in $M^{n, 2}(\Omega)$ (see, for instance,
\cite{Evans}, Theorem 5.3.1 and appendix C.4).

Choose an arbitrary function $\zeta \in C^{\infty}_0(\mathbb{R})$ such that $\zeta(x) = 1$ when $x \in [0, 1]$,
$\zeta(x) = 0$, when $x \ge 2$ and $0 \le \zeta(x) \le 1$ for all $x \in \mathbb{R}$. For any $m \in \mathbb{N}$ define
the function
$$
  u_m(x) = (T v)(x) \prod_{k = 1}^n \big[ 1 - \zeta(m(x_k - a_k)) \big]\big[ 1 - \zeta(m(b_k - x_k)) \big]
  \quad \forall x \in \Omega.
$$
It is clear that $u_m \in M^{n, 2}(\Omega)$ and $u_m = 0$ outside $\prod_{k = 1}^n [a_k + 1 / m, b_k - 1/m]$ for
$m \in \mathbb{N}$ large enough. Let us show that $u_m \to u$ in $M^{n, 2}(\Omega)$, then we get the desired result.

Indeed, it is easily to verify that $u_m \to u$ in $L_2(\Omega)$. Fix an arbitrary $\alpha \in I_r$, $1 \le r \le n$.
Observe that
\begin{multline*}
  \prod_{k = 1}^n \big( 1 - \zeta(m(x_k - a_k)) \big)\big( 1 - \zeta(m(b_k - x_k)) \big) = \\
  = 1 + \sum_{k = 1}^n \zeta(m(x_k - a_k)) \omega_k^{(1)}(x) + \sum_{k = 1}^n \zeta(m(b_k - x_k)) \omega_k^{(2)}(x),
\end{multline*}
where the functions $\omega_k^{(i)}$ are infinitely continuously differentiable and bounded on $\Omega$.
Therefore by Remark~\ref{Rmrk_Multip} one has
\begin{multline*}
  D^{\alpha} u_m(x) = 
  (D^{\alpha} u)(x) + \sum_{k = 1}^n (D^{\alpha} u)(x) \zeta(m(x_k - a_k)) \omega_k^{(1)}(x) + \\
  + \sum_{k = 1}^n (D^{\alpha} u)(x) \zeta(m(b_k - x_k)) \omega_k^{(2)}(x) + 
  \sum_{\beta + \gamma = \alpha, \beta \ne \alpha} (D^{\beta} u)(x) \zeta_{\gamma}(x),
\end{multline*}
where $\zeta_{\gamma} \in C^{\infty}(\Omega)$ and 
\begin{equation} \label{derivFactor}
  \zeta_{\gamma}(x) = 
  D^{\gamma} \Big( \prod_{k = 1}^n \big[1 - \zeta(m(x_k - a_k) \big]\big[ 1 - \zeta(m(b_k - x_k) \big] \Big).
\end{equation}
Hence there exists $C \ge 0$ depending only on $n$, $\Omega$ and $\zeta$ such that
\begin{multline*}
  \int_{\Omega} |D^{\alpha} u_m - D^{\alpha} u|^2 dx \le 
  C \sum_{k = 1}^n \int_{a_k}^{a_k + 2 / m} \int_{\Omega_k} |D^{\alpha} u|^2 \, d \xi^{k} d x_k + \\
  + C \sum_{k = 1}^n \int_{b_k - 2 / m}^{b_k} \int_{\Omega_k} |D^{\alpha} u|^2 \, d \xi^{k} d x_k + \\
  + C \sum_{\beta + \gamma = \alpha, \beta \ne \alpha} 
  \int_{\Omega} |D^{\beta} u|^2 (x) |\zeta_{\gamma}|^2(x) \, dx = A_1(m) + A_2(m) + A_3(m)
\end{multline*}
by virtue of the fact that $\zeta(m(x_k - a_k)) = 0$ for any $x_k > a_k + 2 / m$ and $\zeta(m(b_k - x_k)) = 0$ for any 
$x < b_k - 2 / m$. Here $\Omega_k = \prod_{i \ne k} (a_i, b_i)$ and 
$\xi^k = (\xi_1, \ldots, \xi_{k-1}, x_k, \xi_{k+1}, \ldots, \xi_n)$. It is clear that $A_1(m) \to 0$ and $A_2(m) \to 0$
as $m \to \infty$. Let us show that $A_3(m) \to 0$ as $m \to \infty$, then $\| D^{\alpha} u_m - D^{\alpha} u\|_2 \to 0$
as $m \to \infty$ for any $\alpha \in I_r$, $1 \le r \le n$ and, consequently, $u_m \to u$ in $M^{n, 2}(\Omega)$.

Fix an arbitrary $0 \le k \le r - 1$, $\beta \in I_k$ and $\gamma \in I_{r - k}$ such that $\beta + \gamma = \alpha$.
Without loss of generality we can suppose that $\alpha_1 = \ldots = \alpha_r = 1$, 
$\gamma_1 = \ldots = \gamma_{r - k} = 1$ and $\beta_{r - k + 1} = \ldots = \beta_r = 1$. Our aim is to show that
$$
  \int_{\Omega} |D^{\beta} u|^2 (x) |\zeta_{\gamma}(x)|^2 \, dx \to 0 \text{ as } m \to \infty,
$$
then $A_3(m) \to 0$ as $m \to \infty$. Taking into account \eqref{derivFactor} one gets that
$$
  \zeta_{\gamma}(x) = 
  m^{|\gamma|} \sum_{\eta + \theta = \gamma} \omega_{\eta, \theta}(x) 
  \prod_{l = 1}^n \zeta'( m(x_l - a_l) )^{\eta_l} \prod_{s = 1}^n \zeta'( m(b_s - x_s) )^{\theta_s},
$$
where $\omega_{\eta, \theta} \in C^{\infty}(\Omega)$ and $|\omega_{\eta, \theta}| \le 1$. Therefore it is
sufficient to show that for any multi-indices $\eta$ and $\theta$ such that $\eta + \theta = \gamma$ one has
$$
  m^{2 |\gamma|} \int_{\Omega} |D^{\beta} u|^2 (x) \prod_{l = 1}^n |\zeta'( m(x_l - a_l) )|^{2 \eta_l} 
  \prod_{s = 1}^n |\zeta'( m(b_s - x_s) )|^{2 \theta_s } \, dx \to 0
$$
as $m \to \infty$. Fix an arbitrary $0 \le j \le r - k$, $\eta \in I_j$ and $\theta \in I_{r - k - j}$ such that 
$\eta + \theta = \gamma$. Without loss of generality we can suppose that $\eta_1 = \ldots = \eta_j = 1$ and 
$\theta_{j + 1} = \ldots = \theta_{r - k} = 1$. As it was mentioned above, 
$D^{\beta} u = D^{\beta} Tv = T_{\overline{\beta}}v$ (see~\eqref{DerOfIntegOper}). Taking into account the fact that 
$v \in L_0$ one obtains that $\int_{a_i}^{x_i} v \, d \xi_i = - \int_{x_i}^{b_i} v \, d \xi_i$ for any $1 \le i \le n$.
Hence for a.~e. $x \in \Omega$ one has
$$
  D^{\beta} u = (-1)^{|\theta|} \int_{a_1}^{x_1} \ldots \int_{a_j}^{x_j} 
  \int_{x_{j + 1}}^{b_j} \ldots \int_{x_{r - k}}^{b_{r - k}} \int_{a_{r+1}}^{x_{r+1}} \ldots \int_{a_n}^{x_n} v \,
  d \xi_{n} \ldots d \xi_{r + 1} d \xi_{r-k} \ldots d \xi_1.
$$
Consequently, with the use of the H\"older inequality one gets that there exists $C \ge 0$, depending only on $n$ and
$\Omega$, such that for a.~e. $x \in \Omega$
$$
  |D^{\beta} u|^2 (x) \le C \prod_{l = 1}^j (x_l - a_l) \prod_{s = j + 1}^{r - k} |b_s - x_s|
  (T_{\overline{\beta}} |v|^2)(x).
$$
Therefore, applying the fact that $\zeta'( m(x_k - a_k) ) = 0$ for any $x_k > a_k + 2 / m$ and 
$\zeta'( m(b_k - x_k) ) = 0$ for any $x_k < b_k - 2 / m$ one gets that for some $C_1, C_2 > 0$ that do not depend on
$m$ the following holds
\begin{multline*}
  m^{2|\gamma|} \int_{\Omega} |D^{\beta} u|^2 (x) \prod_{l = 1}^j |\zeta'( m(x_l - a_l) )| 
  \prod_{s = j + 1}^{r - k} |\zeta'( m(b_s - x_s) )| \, dx \le \\
  \le C_1 m^{2|\gamma|} \int_{a_1}^{a_1 + 2 / m} \ldots \int_{a_j}^{a_j + 2 / m}
  \int_{b_{j+1} - 2 / m}^{b_{j+1}} \ldots \int_{b_{r - k} - 2 / m}^{b_{r - k}} 
  \prod_{l = 1}^j (x_l - a_l) \times \\
  \times \prod_{s = j + 1}^{r - k} |b_s - x_s| 
  \Big( (S_{\overline{\gamma}} T_{\overline{\beta}} |v|^2)(x) \Big)
  d x_{r-k} \ldots d x_1 \le \\
  \le C_2 \int_{a_1}^{a_1 + 2 / m} \ldots \int_{a_j}^{a_j + 2 / m} 
  \int_{b_{j + 1} - 2 / m}^{b_{j + 1}} \ldots \int_{b_{r - k} - 2 / m}^{b_{r - k}}
  \int_{a_{r - k + 1}}^{b_{r - k + 1}} \ldots \int_{a_n}^{b_n} |v|^2 \, dx \to 0
\end{multline*}
as $m \to \infty$. Thus, the proof is complete.
\end{proof}

\begin{corollary}
The operator $T$ is an isometric isomorphism between 
$$
  L_0 = \Big\{ v \in L_2(\Omega) \Bigm| \int_{a_i}^{b_i} v d \xi_i = 0, 1 \le i \le n \Big\}
$$
and $M^{n, 2}_0(\Omega)$. Furthermore, the inverse operator of the restriction of $T$ to $L_0$ is the restriction of
the differential operator $D^{(1,\ldots,1)}$ to $M^{n, 2}_0(\Omega)$.
\end{corollary}

\begin{remark}
{(i) The proof of the converse statement of the theorem above is based on the proof of the theorem on
the characterization of $W^{1, p}_0(\Omega)$ in terms of the trace operator (see~\cite{Evans}, theorem 5.5.2).
}

\noindent{(ii) Denote by $Tr$ the trace operator defined on $W^{1, 2}(\Omega)$. Clearly, if 
$u \in M^{n, 2}_0(\Omega) \subset W^{1,2}_0(\Omega)$, then $Tr(u) = 0$ (see~\cite{Evans}, Theorem~5.5.2 and
\cite{Leoni}, Theorem~15.29). However, there is an open question whether the converse statement is true, i.e. 
whether $u \in M^{n, 2}(\Omega)$ and $Tr(u) = 0$ implies that $u \in M^{n, 2}_0(\Omega)$. This question becomes more
reasonable after one notes that if $u \in C^n(\overline{\Omega})$ and $Tr(u) = 0$, i.e. $u|_{\partial \Omega} = 0$, then
$u \in M^{n, 2}_0(\Omega)$. Indeed, from the equality $Tr(u) = 0$ it follows that
\begin{equation} \label{NewtLebFormWhenTr0}
  u(x) = \int_{a_1}^{x_1} \frac{\partial u}{\partial x_1}(\xi_1, x_2, \ldots, x_n) \, d \xi_1 
  \quad \forall x \in \Omega.
\end{equation}
Observe that $u(x_1, a_2, x_3, \ldots, x_n) \equiv 0$. Therefore
$\partial u / \partial x_1 (x_1, a_2, x_3, \ldots, x_n) = 0$ for all $x \in \Omega$. Hence 
$$
  \frac{\partial u}{\partial x_1}(x) = \int_{a_2}^{x_2}
  \frac{\partial^2 u}{\partial x_2 \partial x_2} (x_1, \xi_2, x_3, \ldots, x_n) \, d \xi_2 
  \quad \forall x \in \Omega.
$$
Taking into account \eqref{NewtLebFormWhenTr0} one gets that
$$
  u(x) = \int_{a_1}^{x_1} \int_{a_2}^{x_2}
  \frac{\partial^2 u}{\partial x_2 \partial x_2} (\xi_1, \xi_2, x_3, \ldots, x_n) \, d \xi_2 d \xi_1
  \quad \forall x \in \Omega.
$$
Applying mathematical induction one can easily obtain that $u = T D^{(1,\ldots,1)} u$ and 
$D^{(1, \ldots, 1)} u \in L_0$, which by the theorem above yields that $u \in M^{n, 2}_0(\Omega)$.
}
\end{remark}

In the following sections, we will use the function space $M^{m,p}(\Omega, \mathbb{R}^d)$ 
(or $M^{m,p}_0(\Omega, \mathbb{R}^d)$) that consists of all functions 
$u = (u_1, \ldots, u_d) \colon \Omega \to \mathbb{R}^d$ such that $u_i \in M^{m, p}(\Omega)$ 
(or $u_i \in M_0^{m, p}(\Omega)$) for all $i \in \{ 1, \ldots, d \}$. It is clear that the spaces 
$M^{m,p}(\Omega, \mathbb{R}^d)$ and $M^{m,p}_0 (\Omega, \mathbb{R}^d)$ possess the same properties as their 
``one dimensional'' counterparts. In particular, one can easily obtain a characterization of the space 
$M^{n,2}_0(\Omega, \mathbb{R}^d)$ similar to the characterization of $M^{n,2}_0(\Omega)$.

\section{The Direction of Steepest Descent}
\label{Section_NewMinMethods}

In this section, we compute the direction of steepest descent of the main functional of the calculus of variations.
This direction can be utilized in order to design new direct numerical methods for solving multidimensional problems of
the calculus of variations. 

Consider the following problem of the calculus of variations
\begin{equation} \label{MainProblem}
  \min \: \mathcal{I}(u) = \int_{\Omega} f(x, u(x), \nabla u(x)) \, dx \quad
  \text{subject to} \quad u \in \overline{u} + M^{n, 2}_0(\Omega, \mathbb{R}^d),
\end{equation}
where $\overline{u} \in W^{1, 2}(\Omega, \mathbb{R}^d)$ is a given function that defines boundary conditions, 
$u = (u_1, \ldots, u_d)$ and 
$$
  \nabla u = \left\{ \frac{\partial u_j}{\partial x_i} \right\}_{1 \le i \le n}^{1 \le j \le d}
  \in \mathbb{R}^{d \times n}.
$$
We suppose that the function $f \colon \Omega \times \mathbb{R}^d \times \mathbb{R}^{d \times n} \to \mathbb{R}$,
$f = f(x, u, z)$ satisfies the Carath\'eodory condition, is differentiable with respect to $u$ and $z$, and the
derivatives $\partial f / \partial u_i$, $\partial f / \partial z_{ij}$ satisfy the Carath\'eodory condition as well.
Let also the following growth conditions be valid: 
\begin{enumerate}
\item{there exist $C > 0$ and $g \in L_1(\Omega)$ such that for a.~e. $x \in \Omega$ and for all 
$u \in \mathbb{R}^n$, $z \in \mathbb{R}^{d \times n}$ one has $|f(x, u, z)| \le C( |u|^2 + |z|^2 ) + g(x)$.
}
\item{there exist $D_1, D_2 > 0$ and $g_1, g_2 \in L_2(\Omega)$ such that for a.~e. $x \in \Omega$ and for all 
$u \in \mathbb{R}^n$, $z \in \mathbb{R}^{d \times n}$ one has
$$
  \left| \frac{\partial f}{\partial u}(x, u, z) \right| \le D_1 ( |u| + |z| ) + g_1(x), \quad
  \left| \frac{\partial f}{\partial z}(x, u, z) \right| \le D_2 ( |u| + |z| ) + g_2(x).
$$
}
\end{enumerate}

\begin{remark}
In the general case, problem \eqref{MainProblem} is not equivalent to the standard problem of the calculus of
variations
\begin{equation} \label{StandProblem}
  \min \: \mathcal{I}(u) = \int_{\Omega} f(x, u(x), \nabla u(x)) \, dx \quad
  \text{subject to} \quad u \in \overline{u} + W^{1, 2}_0(\Omega, \mathbb{R}^d).
\end{equation}
Problems \eqref{MainProblem} and \eqref{StandProblem} are equivalent if and only if there exists an optimal
solution $u^*$ of problem \eqref{StandProblem} such that $u^* \in \overline{u} + M^{n, 2}_0(\Omega, \mathbb{R}^d)$.
In other words, problems \eqref{MainProblem} and \eqref{StandProblem} are equivalent if and only if there exists a
``sufficiently smooth'' optimal solution $u^*$ of problem \eqref{StandProblem}.
\end{remark}

It is easy to see that the functional $\mathcal{I}$ is G\^ateaux differentiable at every point 
$u \in W^{1, 2}(\Omega, \mathbb{R}^d)$ and its G\^ateaux derivative has the form
\begin{equation} \label{FuncGrad}
  \mathcal{I}'[u](h) = \int_{\Omega} 
  \left( \left\langle \frac{\partial f}{\partial u}(x, u(x), \nabla u(x)), h(x) \right\rangle +
  \left\langle \frac{\partial f}{\partial z}(x, u(x), \nabla u(x)), \nabla h(x) \right\rangle \right) dx
\end{equation}
for all $h \in W^{1, 2}(\Omega, \mathbb{R}^d)$. Here $\langle \cdot, \cdot \rangle$ is the inner product in
$\mathbb{R}^s$, and
$$
  \frac{\partial f}{\partial u} = 
  \left(\frac{\partial f}{\partial u_1}, \ldots, \frac{\partial f}{\partial u_d}\right), \quad
  \frac{\partial f}{\partial z} = \left\{ \frac{\partial f}{\partial z_{ji}} \right\}_{1 \le i \le n}^{1 \le j \le d}.
$$
As it was mentioned above, the problem of finding the direction of steepest descent of the functional $\mathcal{I}$ is
very complicated. We utilise the characterization of $M^{n, 2}_0(\Omega, \mathbb{R}^d)$ (Theorem \ref{ThCharctOfMd2}) in
order to solve this problem and to design new direct numerical methods for minimizing this functional. 

Recall that the operator $T$,
$$
  (T v)(x) = \int_{a_n}^{x_n} \ldots \int_{a_1}^{x_1} v(\xi_1, \ldots, \xi_n) d \xi_1 \ldots d \xi_n 
  \quad \text{for a.e. } x \in \Omega,
$$
is an isometric isomorphism between $L_0 \subset L_2(\Omega, \mathbb{R}^d)$ and $M^{n, 2}_0(\Omega, \mathbb{R}^d)$.
Therefore problem \eqref{MainProblem} is equivalent to the problem
\begin{align*}
  {}&\min \quad F(v) = \mathcal{I}(\overline{u} + T v) = 
  \int_{\Omega} f(x, \overline{u}(x) + (T v)(x), \nabla \overline{u}(x) + \nabla (T v)(x)) \, dx \\
  {}&\text{subject to} \quad v \in L_0,
\end{align*}
We suppose that the functional $F$ is defined on $L_0$, and will minimize $F$ over the space $L_0$.

The functional $F$ is G\^ateaux differentiable, and integrating by parts in \eqref{FuncGrad} one gets that the
G\^ateaux derivative of the functional $F$ has the form
$$
  F'[v](h) = \langle Q(u), h \rangle = \int_{\Omega} \langle Q(u)(x), h(x) \rangle \, dx \quad \forall h \in L_0,
$$
where $u = \overline{u} + T v$,
\begin{multline} \label{QofUdef}
  Q(u)(x) = (-1)^n \int_{x_n}^{b_n} \ldots \int_{x_1}^{b_1}
  \frac{\partial f}{\partial u}(\xi, u(\xi), \nabla u(\xi))\, d\xi_1 \ldots d\xi_n + \\
  + (-1)^{n-1} \sum_{i = 1}^n \int_{x_n}^{b_n} \ldots \int_{x_{i + 1}}^{b_{i + 1}} \int_{x_{i - 1}}^{b_{i - 1}} \ldots
  \int_{x_1}^{b_1} \\
  \frac{\partial f}{\partial z_i}(\xi^i, u(\xi^i), \nabla u(\xi^i)) \, 
  d \xi_1 \ldots d \xi_{i-1} d \xi_{i+1} \ldots d \xi_n,
\end{multline}
$\partial f / \partial z_i = (\partial f / \partial z_{1i}, \ldots, \partial f / \partial z_{di})$,
and $\xi^i = (\xi_1, \ldots, \xi_{i - 1}, x_i, \xi_{i + 1}, \ldots, \xi_n)$. Note that $Q(u) \in L_2(\Omega)$ for any 
$u \in M^{n, 2}_0(\Omega, \mathbb{R}^d)$, since the derivatives of $f$ satisfy the growth condition. Hence one has that
$$
  F'[v](h) = \langle G(u), h \rangle = \int_{\Omega} \langle G(u)(x), h(x) \rangle \, dx \quad \forall h \in L_0,
$$
where $u = \overline{u} + Tv$, and $G(u) = Pr_{L_0} Q(u)$ is the projection of $Q(u)$ onto $L_0$ in 
$L_2(\Omega, \mathbb{R}^d)$ (see~Proposition~\ref{Prp_AnalyticalProjection} above). Note that $G(u) \in L_0$ is the
G\^ateaux gradient of the functional $F$ at the point $v$. Consequently, $- G(u) / \| G(u) \|_2$ is the direction of
steepest descent of $F$ at the point $v$, which implies that $- T G(u) / \| G(u) \|_2$ is the direction of steepest
descent of the functional $\mathcal{I}$ at the point $u = \overline{u} + T v$.

Thus, the following result holds true.

\begin{proposition}
Let $v \in L_0$ and $u = \overline{u} + T v$ (or, equivalently, $u = \overline{u} + w$ for some 
$w \in M^{n, 2}_0(\Omega, \mathbb{R}^d)$ and $v = D^{(1, \ldots, 1)} w$) be such that $F'[v] \ne 0$. Let also 
$\mathcal{J}(w) = \mathcal{I}(\overline{u} + w)$ for any $w \in M^{n, 2}_0(\Omega, \mathbb{R}^d)$. Then 
$- T G(u) / \| G(u) \|_2$ is the direction of steepest descent of $\mathcal{J}$ at the point $w = T v$,
and $- G(u) / \| G(u) \|_2$ is the direction of steepest descent of $F$ at the point $v$.
\end{proposition}

\begin{remark} \label{Remark_SDD_Preserves_Trace}
Note that for any $v \in L_0$ one has 
$u = Tv \in M^{n, 2}_0 (\Omega, \mathbb{R}^d) \subset W^{1, 2}_0(\Omega, \mathbb{R}^d)$. Therefore, as it was mentioned
above, the trace $Tr(u)$ of the function $u = Tv$ is correctly defined and equals zero. Thus, the direction of steepest
descent $- T G(u) / \| G(u) \|_2$ does not change the values of a function on the boundary of $\Omega$, i.e. 
$Tr( u + \alpha T G(u) ) = Tr(u)$ for all $\alpha \in \mathbb{R}$.
\end{remark}

Let us mention several simple properties of the mappings $G(\cdot)$ and $Q(\cdot)$.

\begin{proposition} \label{PrpPropertOfGrad}
Let $v \in L_0$ and $u = \overline{u} + T v$ (or, equivalently, $u = \overline{u} + w$ for some 
$w \in M^{n, 2}_0(\Omega, \mathbb{R}^d)$ and $v = D^{(1, \ldots, 1)} w$). Let also $\mathcal{J}(w) =
\mathcal{I}(\overline{u} + w)$ for any $w \in M^{n, 2}_0(\Omega, \mathbb{R}^d)$. Then the following statements hold
true:
\begin{enumerate}
\item{$\| \mathcal{J}'[w] \|^2 = \| F'[v] \|^2 = \mathcal{I}'[u](TG(u)) = F'[v](G(u))  = \big( \| G(u) \|_2 \big)^2$;
}

\item{$\mathcal{I}'[u] = 0$ iff $F'[v] = 0$ iff $D^{(1, \ldots, 1)} Q(u) = 0$ in the weak sense;
} 

\item{suppose that $u \in C^2(\Omega, \mathbb{R}^d)$ and 
$f \in C^2(\Omega \times \mathbb{R}^d \times \mathbb{R}^{d \times n})$. Then
$F'[v] = 0$ if and only if for all $x \in \Omega$ one has
$$
  \frac{\partial f}{\partial u}(x, u(x), \nabla u(x)) - 
  \sum_{i = 1}^n \frac{\partial}{\partial x_i} \frac{\partial f}{\partial z_i} (x, u(x), \nabla u(x)) = 0,
$$
i.e. $F'[v] = 0$ iff $u = \overline{u} + T v$ satisfies the Euler--Lagrange equation. 
}
\end{enumerate}
\end{proposition}

\begin{proof}
The validity of the first statement follows directly from definitions. Let us prove the second statement. It is clear
that $\mathcal{I}'[u](Th) = F'[v](h)$ for any $h \in L_0$. Therefore $\mathcal{I}'[u] = 0$ if and only if
\begin{equation} \label{WeakEulerLagrangeEq}
  F'[v](h) = \langle Q(u), h \rangle = \int_{\Omega} \langle Q(u)(x), h(x) \rangle \, dx = 0 \quad \forall h \in L_0.
\end{equation}
From the fact that for any $\varphi \in C^{\infty}_0(\Omega, \mathbb{R}^d)$ one has 
$D^{(1, \ldots, 1)} \varphi \in L_0$ it follows that \eqref{WeakEulerLagrangeEq} is equivalent to the fact that 
$D^{(1, \ldots, 1)} Q(u) = 0$ in the weak sense.

Suppose in addition that $u \in C^2(\Omega, \mathbb{R}^d)$ and 
$f \in C^2(\Omega \times \mathbb{R}^d \times \mathbb{R}^{d \times n})$. Then integrating in \eqref{WeakEulerLagrangeEq}
by parts one gets that
$$
  \int_{\Omega} \langle g(x), \varphi(x) \rangle \, dx = 0 \quad 
  \forall \varphi \in C_0^{\infty}(\Omega, \mathbb{R}^d),
$$
where 
$$
  g(x) = (D^{(1, \ldots, 1)} Q(v))(x)
  = \frac{\partial f}{\partial u}(x, u(x), \nabla u(x))
  - \sum_{i = 1}^n \frac{\partial}{\partial x_i} \frac{\partial f}{\partial z_i} (x, u(x), \nabla u(x)).
$$
Applying the fundamental lemma of the calculus of variations (see, e.g., \cite{Dacorogna}, Theorem~3.40) one obtains
the desired result.
\end{proof}

Observe that the mapping $u \to G(u) \in L_0$ is well defined defined for any $u \in W^{1, 2}(\Omega, \mathbb{R}^d)$.

\begin{proposition}
Let $u^* \in \overline{u} + W^{1, 2}_0(\Omega, \mathbb{R}^d)$. Then $\mathcal{I}'[u^*](\cdot) = 0$ on
$W^{1, 2}_0(\Omega, \mathbb{R}^d)$ if and only if $G(u^*) = 0$.
\end{proposition}

\begin{proof}
Necessity. Suppose that $\mathcal{I}'[u^*](\cdot) = 0$ on $W^{1, 2}_0(\Omega, \mathbb{R}^d)$. Note that 
$T G(u^*) \in M^{n, 2}_0(\Omega, \mathbb{R}^d) \subset W^{1, 2}_0(\Omega, \mathbb{R}^d)$, since $G(u^*) \in L_0$.
Hence $\mathcal{I}'[u^*](T G(u^*)) = 0$. Consequently, integrating by parts one gets
\begin{multline*}
  0 = \mathcal{I}'[u^*](- TG(u^*)) = - \int_{\Omega} \langle Q(u^*)(x), G(u^*)(x) \rangle \, dx = \\
  = - \int_{\Omega} \langle G(u^*)(x), G(u^*(x) \rangle \, dx = - (\| G(u^*) \|_2)^2,
\end{multline*}
since $G(u^*) \in L_0$ and $Pr_{L_0} Q(u^*) = G(u^*)$. Thus, $G(u^*) = 0$.

Sufficiency. Let $G(u^*) = 0$. Then it is easy to verify that $\mathcal{I}'[u^*](h) = 0$ for all 
$h \in M^{n, 2}_0(\Omega, \mathbb{R}^d)$. The space $M^{n, 2}_0(\Omega, \mathbb{R}^d)$ is dense in 
$W^{1, 2}_0(\Omega, \mathbb{R}^d)$, and the linear functional $\mathcal{I}'[u^*]$ is bounded on 
$W^{1, 2}(\Omega, \mathbb{R}^d)$. Therefore $\mathcal{I}'[u^*](h) = 0$ for any 
$h \in W^{1, 2}_0(\Omega, \mathbb{R}^d)$.
\end{proof}

Since we now know the gradient $G(\cdot)$ of the functional $F$, we can modify most of the gradient--based methods of
finite--dimensional optimization to the case of this functional and use these methods in order to find critical points
(or points of global minimum in the convex case) of the functionals $F$ and $\mathcal{I}$. Convergence analysis of these
methods can be performed in the standard way (see, e.g.,
\cite{Rosenbloom,Polyak,Daniel,DemRub70,Vainberg2,KantAkilov,Penot} for
convergence analysis of minimization methods in Banach and Hilbert spaces.).

\section{Conclusion}

In this article we developed a new approach to the design of minimization algorithms for the main
problem of the calculus of variations. Let us discuss some possible generalizations of this approach and some directions
of future research.

\subsection{More General Boundary Conditions}

The results developed in this article can be modified to the case when the boundary condition has the form 
$u|_{\Gamma} = \psi$, where $\Gamma \subset \partial \Omega$. In other words, one can extend the theory presented in
this article to the case when the values of a function $u$ are specified only on a part of the boundary of $\Omega$.
However, it should be mentioned that any modification of the methods discussed above to the case of more general
boundary conditions requires a different formalization, then the one based on the use of the space $M^{n, 2}_0(\Omega)$.
Therefore the following discussion has an informal character.

Let, for example, $n = 2$ and the boundary condition has the from
$$
  u|_{\Gamma} = \psi, \quad \Gamma = \partial \Omega \setminus [a_1, b_1] \times \{ b_2 \},
$$
i.e. the values of $u$ are not specified on the upper side on the rectangle $\Omega = (a_1, b_1) \times (a_2, b_2)$.
Then one can show that the gradient of the functional $F$ has the form
$$
  Q(u)(x_1, x_2) - \frac{1}{b_1 - a_1} \int_{a_1}^{b_1} Q(u)(\xi_1, x_2) \, d \xi_1 \quad 
  \forall (x_1, x_2) \in \Omega.
$$
Here we use the same notation as in Section~\ref{Section_NewMinMethods}.

In the case when the boundary condition has the form
$$
  u|_{\Gamma} = \psi, \quad \Gamma = 
  \big( \{a_1 \} \times [a_2, b_2] \big) \cup \big( [a_1, b_1] \times \{ a_2 \} \big),
$$
the gradient of the functional $F$ coincides with the function $Q(u)$.

If the  boundary condition has the from
$$
  u|_{\Gamma} = \psi, \quad \Gamma = 
  \big( \{a_1 \} \times [a_2, b_2] \big) \cup \big( [a_1, b_1] \times \{ b_2 \} \big),
$$
then one should use a different representation of a function $u$:
$$
  u(x_1, x_2) = - \int_{a_1}^{x_1} \int_{x_2}^{b_2} 
  \frac{\partial^2 u}{\partial x_1 \partial x_2}(\xi_1, \xi_2) d \xi_2 d \xi_1.
$$
This representation can be used to compute the direction of steepest descent.

\subsection{Isoperimetric Problems}

One can easily modify the proposed approach to the case of problems of the calculus of variations with linear
isoperimetric constraints. Namely, let $n = 2$, and suppose that there is the additional constraint
\begin{equation} \label{IsoperimConstr}
  \mathcal{J}(u) = \int_{\Omega} \big( g_0(x) u(x) + g_1(x) D_1 u(x) + g_2(x) D_2 u(x) \big) \, dx = c.
\end{equation}
Denote by $Q_{\mathcal{I}}(u)$ the function $Q(u)$ for the functional $\mathcal{I}$ (see \eqref{QofV} and
\eqref{QofUdef}). 

Applying the same argument as in Section~\ref{Section_InformalArg} one can easily demonstrate that the direction of
steepest descent for the problem with additional isoperimetric constraint~\eqref{IsoperimConstr} has the same form
as for the problem without this constrain with the function $Q_{\mathcal{I}}(u)(x)$ replaced by the function
$Q_{\mathcal{I}}(u) + \lambda Q_{\mathcal{J}}(u)$. Here $\lambda$ is a constant that is chosen so that the direction of
steepest descent satisfies constraint~\eqref{IsoperimConstr} with $c = 0$.

\subsection{Functionals Depending on Higher Order Derivatives}

Let us also note that the approach developed in this article can be generalized to the case when the functional
$\mathcal{I}(u)$ depends on derivatives of the function $u$ of order greater than $1$. Indeed, consider, for instance,
the following two-dimensional problem of the calculus of variations:
\begin{align} \label{HigherOrderProblem}
  {}&\min \quad \mathcal{I}(u) = 
  \int_{\Omega} f\big(x, u(x), \nabla u(x), \nabla^2 u(x) \big) \, dx \\
  {}&\text{subject to} \quad u|_{\partial \Omega} = \psi_1, \quad 
  \frac{\partial u}{\partial \nu} \Big|_{\partial \Omega} = \psi_2, \label{BoundaryCondHigherOrder}
\end{align}
where $\Omega = (a_1, b_1) \times (a_2, b_2)$, and $\nu$ is the outward unit normal at the boundary of $\Omega$. Let us
demonstrate how one can easily transform this problem utilizing the same technique as above in order to easily compute
the direction of steepest descent of the functional $\mathcal{I}(u)$ with respect to a certain norm. Here we provide
only an informal description of such transformation. A formalization of this transformation can be done in the same way
as in the case of the functional that depends only on the first order derivatives.

Let sufficiently smooth functions $u$ and $\overline{u}$ satisfy boundary conditions \eqref{BoundaryCondHigherOrder}.
Then the function $w = u - \overline{u}$ satisfies the same boundary conditions with 
$\psi_1(\cdot) = \psi_2(\cdot) = 0$. Hence
$$
  w(x_1, x_2) = \int_{a_1}^{x_1} \frac{\partial w}{\partial x_1}(\xi_1, x_2) \, d \xi_1
$$
due to the fact that $w(a_1, \cdot) = 0$. Then applying the fact that $w'_{x_1}(a_1, \cdot) = 0$ one obtains that
$$
  w(x_1, x_2) = \int_{a_1}^{x_1} \int_{a_1}^{\xi_1} 
  \frac{\partial^2 w}{\partial x_1^2}(\theta_1, x_2) \, d \theta_1 \, d \xi_1.
$$
Note that $w''_{x_1 x_1}(\cdot, a_2) = 0$, since $w(\cdot, a_2) = 0$, which implies that
$$
  w(x_1, x_2) = \int_{a_1}^{x_1} \int_{a_1}^{\xi_1} \int_{a_2}^{x_2}
  \frac{\partial^3 w}{\partial x_2 \partial x_1^2}(\theta_1, \xi_2) \, d \xi_2 \, d \theta_1 \, d \xi_1.
$$
Finally, from the fact that $w'_{x_2}(\cdot, a_2) = 0$ it follows that 
$w'''_{x_2 x_1 x_1}(\cdot, a_2) = 0$, which yields
$$
  w(x) = \big(T v \big)(x), \quad v = \frac{\partial^4 w}{\partial x_2^2 \partial x_1^2},
$$
where
\begin{equation} \label{HigherOrderIntegralOperator}
  \big( T v \big)(x) = \int_{a_1}^{x_1} \int_{a_1}^{\xi_1} \int_{a_2}^{x_2} \int_{a_2}^{\xi_2}
  v(\theta_1, \theta_2) \, d \theta_2 \, d \xi_2 \, d \theta_1 \, d \xi_1.
\end{equation}
Furthermore, it is easy to check that
$$
  w|_{\partial \Omega} = 0, \quad \frac{\partial w}{\partial \nu} \Big|_{\partial \Omega} = 0
$$
if and only if
\begin{equation} \label{HigherOrderConstraints}
  \int_{a_1}^{b_1} v(\xi_1, \cdot) \, d \xi_1 = 0, \quad \int_{a_2}^{b_2} v(\cdot, \xi_2) \, d \xi_2 = 0.
\end{equation}
Thus, the following result holds true (cf.~Proposition~\ref{PrpChrctFunc}).

\begin{proposition}
Let $u \colon [a_1, b_1] \times [a_2, b_2] \to \mathbb{R}$ be a sufficiently smooth function. Then $u$ satisfies
boundary conditions \eqref{BoundaryCondHigherOrder} if and only if there exists a sufficiently smooth function $v$ such
that
\begin{enumerate}
\item{$u = \overline{u} + T v$, where the operator $T$ is defined in \eqref{HigherOrderIntegralOperator};}

\item{$\int_{a_1}^{b_1} v(\xi_1, \cdot) \, d \xi_1 \equiv 0$ and 
$\int_{a_2}^{b_2} v(\cdot, \xi_2) \, d \xi_2 \equiv 0$.
}
\end{enumerate}
Moreover, $v = \partial^4 u / \partial x_1^2 \partial x_2^2$.
\end{proposition}

Applying the proposition above one obtains that problem \eqref{HigherOrderProblem}, \eqref{BoundaryCondHigherOrder} is
equivalent to the problem of minimizing the functional $F(v) = \mathcal{I}(\overline{u} + T v)$ subject to 
linear equality constraints \eqref{HigherOrderConstraints}. The direction of steepest descent for this problem with
respect to the $L_2$-norm can be easily computed in the same way as in the proof of Proposition~\ref{Prp_SDD_Informal}.

\subsection{Directions of Future Research}

Let us briefly outline some other directions of future research:
\begin{itemize}
\item{One can modify Newton's method \cite{NewtonMethod} to the multidimensional case, and use the methods developed
in this article to perform each iteration of Newton's method. The use of Newton's method might be reasonable in the
case of highly nonlinear problems of the calculus of variations.
}

\item{In the case of problems of the calculus of variations with nonlinear isoperimetric constraints one can modify
sequential quadratic programming methods \cite{SeqQuadProg} in order to solve these problems. On each iteration of this
method one needs to solve a problem with a quadratic functional and linear constraints that can be solved with the use
of the methods discussed above.}

\item{One can use the same augmented Lagrangian method as in \cite{Pedregal} (or some other augmented Lagrangian
methods) in order to apply the approach developed in this article to variational problems with pointwise constraints.
}
\end{itemize}

\section{Acknowledgements}

The author is sincerely grateful to his colleague G.Sh. Tamasyan and the late professor V.F. Demyanov. It should be
noted that this article is a straightforward generalization of their works on the calculus of variations. Also, 
the author wishes to express his thanks to professor M.Z. Nashed for thoughtful and stimulating comments that helped to
improve the quality of the article.

\bibliographystyle{abbrv}  
\bibliography{MinMethodCalcVar_bib}

\end{document}